\documentclass[12pt]{amsart}
\usepackage{amsmath}
\usepackage{amsfonts}
\usepackage{amssymb}
\usepackage{amsthm}
\usepackage{hyperref}

\theoremstyle{plain}
\newtheorem{theorem}{Theorem}[section]
\newtheorem{lemma}[theorem]{Lemma}

\theoremstyle{definition}

\theoremstyle{remark}

\numberwithin{equation}{section}

\author{Paul Jenkins}
\author{Kyle Pratt}
\title{Coefficient Bounds for Level 2 Cusp Forms and Modular Functions}

\begin{document}
\date{}
\maketitle

\begin{abstract}
We give explicit upper bounds for the coefficients of arbitrary weight $k$, level 2 cusp forms, making Deligne's well-known $O(n^{\frac{k-1}{2}+\epsilon})$ bound precise. We also derive asymptotic formulas and explicit upper bounds for the coefficients of certain level 2 modular functions.
\end{abstract}

\section{Introduction}
The Fourier coefficients of modular forms encode interesting arithmetic information. To give just three examples, the Fourier coefficients of modular forms are intimately connected to representations of integers as sums of squares \cite{Rou2}, Galois representations \cite{Rib}, and integer partitions (e.g. \cite{And}, Chapter 5). These rich interactions between modular forms and other branches of mathematics have given them a place of central importance in modern number theory.

It is natural to ask about the size of the coefficients of a modular form. The coefficients of cusp forms in particular have attracted a great deal of attention. Ramanujan \cite{Ram} studied the coefficients of $\Delta(z)$, the unique normalized cusp form of weight 12 for $\text{SL}_2(\mathbb{Z})$ given by
\begin{align*}
\Delta(z) = q\prod_{n=1}^\infty (1-q^n)^{24} = \sum_{n=1}^\infty \tau(n) q^n,
\end{align*}
where as usual $q = e^{2\pi i z}$. Ramanujan conjectured that
\begin{align*}
|\tau(n)| \leq d(n) n^{\frac{11}{2}},
\end{align*}
where $d(n)$ is the number of divisors of $n$. Petersson \cite{Pet1} generalized Ramanujan's conjecture to cusp forms for congruence subgroups of $\text{SL}_2(\mathbb{Z})$. The Ramanujan-Petersson conjectures were proved by Deligne \cite{Del} as a consequence of his work on the Weil conjectures. The corresponding conjectures for Maass forms and automorphic forms on $\text{GL}(n)$ for $n > 2$ remain unresolved (see \cite{BB} for details).

Deligne's result applies to newforms, certain cusp forms that are eigenforms for all of the Hecke operators (see Section \ref{Background section} for more details). For such a weight $k$ newform, Deligne's work implies that the coefficient of $q^n$ is bounded above by $d(n) n^{\frac{k-1}{2}}$. Any cusp form can be written as a linear combination of newforms and newforms acted on by various operators, so it is still the case that the coefficients of a general cusp form $f$ are $O(d(n) n^{\frac{k-1}{2}})$. However, the implied constant depends heavily on $f$, and it is a nontrivial problem to determine this constant (or even the size of this constant). Several researchers have studied the implied constant in $O(d(n) n^{\frac{k-1}{2}})$ for various families of cusp forms (see e.g. \cite{KR}, \cite{Rou1}, \cite{Rou2}). Rouse and the first author \cite{JR} gave an explicit bound on the implied constant for all cusp forms for $\text{SL}_2(\mathbb{Z})$ (earlier work of Chua \cite{Chu} makes Hecke's $O(n^{\frac{k}{2}}$ bound explicit). For a cusp form $G = \sum_{n=1}^\infty a(n)q^n$ of weight $k$ for $\text{SL}_2(\mathbb{Z})$, they proved that
\begin{align}\label{Jenkins Rouse Thm}
|a(n)| &\leq \sqrt{\log k} \left(11 \sqrt{\sum_{m=1}^\ell \frac{|a(m)|^2}{m^{k-1}}} + \frac{e^{18.72}(41.41)^{k/2}}{k^{\frac{k-1}{2}}} \cdot \left|\sum_{m=1}^\ell a(m) e^{-7.288m} \right| \right) \cdot d(n) n^{\frac{k-1}{2}},
\end{align}
where $\ell$ is the dimension of the weight $k$ cusp form space $S_k(\text{SL}_2(\mathbb{Z}))$. The fact that (\ref{Jenkins Rouse Thm}) incorporates the first $\ell$ coefficients of $G$ is natural, since these coefficients uniquely identify $G$ in $S_k(\text{SL}_2(\mathbb{Z}))$.

The main result of this paper makes Deligne's implied constant explicit for weight $k$ cusp forms for $\Gamma_0(2)$. To state our main theorem we define some notation. For a positive even integer $k$, write $k = 4\ell + k'$, where $k' \in \{0,2\}$, hence $\ell = \lfloor \frac{k}{4} \rfloor$. It is convenient to write $k$ in this form because then the dimension of $S_k(2)$ is $\ell - 1$. We define a function $B(k)$ by 
\begin{align*}
B(k) &= \frac{e^{5.449}(6.274)^k}{(\frac{k}{4}-1)^{\frac{k-1}{2}}} + \frac{e^{10.905}(4.793)^k}{\sqrt{(k-2)!}} + \frac{e^{6.511}(10.096)^k}{\sqrt{(k-2)!}}.
\end{align*}
We now state the main result of this paper.

\begin{theorem}\label{Main Theorem}
Let $k \geq 8$ be an even integer, and let $G$ be a cusp form of weight $k$ for $\Gamma_0(2)$. Write
\begin{align*}
G(z) = \sum_{n=1}^\infty a(n)q^n.
\end{align*}
Then
\begin{align*}
|a(n)| &\leq \sqrt{\log k} \left(103 \sum_{m=1}^{\ell-1} \frac{|a(m)|}{m^{\frac{k-1}{2}}} + B(k) \sum_{m=1}^{\ell-1} |a(m)| e^{-7.288 m} \right) d(n) n^{\frac{k-1}{2}}.
\end{align*}
\end{theorem}
\noindent The condition $k \geq 8$ is not a restriction at all, since $S_k(2)= \{0\}$ for $k < 8$. 

The proof of Theorem \ref{Main Theorem} is similar to the proof of (\ref{Jenkins Rouse Thm}). We study the basis of cusp forms for $S_k(2)$ given by
\begin{align*}
F_{k,m}(z) = q^m + \sum_{n = \ell}^\infty A_k(m,n) q^n,
\end{align*}
with $1 \leq m \leq \ell - 1$. This basis is useful because, for any given $G \in S_k(2)$, it is trivial to write $G$ in terms of the basis elements. Hence, Theorem \ref{Main Theorem} follows from suitable bounds on the coefficients of $F_{k,m}$. We write $F_{k,m}$ as
\begin{align*}
F_{k,m} = \sum_i \alpha_i f_i + \sum_j \beta_j g_j,
\end{align*}
where $\alpha_i,\beta_j \in \mathbb{R}$, the $f_i$ are oldforms, and the $g_j$ are newforms. We choose the $f_i$ so that their coefficients are bounded in absolute value by $C d(n) n^{\frac{k-1}{2}}$, for $C \geq 1$ an absolute constant. It follows that the coefficients of $F_{k,m}$ are bounded above by
\begin{align*}
C \left(\sum_i |\alpha_i| + \sum_j |\beta_j| \right) d(n) n^{\frac{k-1}{2}}.
\end{align*}
By the Cauchy-Schwarz inequality, it suffices to get an upper bound on $\sum_i \alpha_i^2 + \sum_j \beta_j^2$.

Letting $\langle \cdot , \cdot \rangle$ denote the Petersson inner product (see Section \ref{Background section} for the definition), we have that
\begin{align*}
\langle F_{k,m}, F_{k,m} \rangle &= \sum_i \alpha_i^2 \langle f_i,f_i \rangle + \sum_j \beta_j^2 \langle g_j,g_j \rangle
\end{align*}
for appropriate choice of $f_i$. Using the fact that $\langle f_i,f_i \rangle, \langle g_j,g_j \rangle$ are multiples of special values of $L$-functions, we obtain a lower bound on $\langle F_{k,m},F_{k,m} \rangle$ of the form
\begin{align*}
\langle F_{k,m}, F_{k,m} \rangle &\geq \left(\sum_i \alpha_i^2 + \sum_j \beta_j^2 \right) h(k),
\end{align*}
for some function $h$. Therefore, we require an upper bound on $\langle F_{k,m}, F_{k,m} \rangle$, which involves bounding several integrals. Bounding the integrals requires upper bounds on the coefficients of $F_{k,m}$ when $F_{k,m}$ is acted on by various matrices in $\text{SL}_2(\mathbb{Z})$. By using the generating function for $F_{k,m}$ given in \cite{GJ}, we obtain bounds of the form 
\begin{align*}
c_1 \cdot c_2^{\ell} e^{c_3 m + c_4 n},
\end{align*}
with $c_1,c_2>0$ and $0 < c_4 < \frac{\sqrt{3}}{2}$. While these bounds are poor, they are sufficient to obtain a reasonable upper bound on $\langle F_{k,m}, F_{k,m} \rangle$.

In addition to cusp forms, we also study the coefficients of weight zero modular functions for $\Gamma_0(2)$. In general, the coefficients of modular functions grow much faster than the coefficients of cusp forms. The classical example of a modular function is the $j$-function, which is modular of weight zero for $\text{SL}_2(\mathbb{Z})$. The $j$-function has a Fourier expansion of the form
\begin{align*}
j(z) = q^{-1} + 744 + \sum_{n=1}^\infty c(n) q^n,
\end{align*}
with the $c(n)$ positive integers. Petersson \cite{Pet2} and Rademacher \cite{Rad} independently obtained an asymptotic formula for $c(n)$. They found that 
\begin{align}\label{j coeff asymp}
c(n) \sim \frac{1}{\sqrt{2} \ n^{3/4}} e^{4\pi \sqrt{n}}.
\end{align}
Since the coefficients of cusp forms are $O(n^{\frac{k-1}{2} + \epsilon})$, we see that the coefficients of $j$ dwarf the coefficients of any cusp form when $n$ is large.

The asymptotic formula (\ref{j coeff asymp}) gives the true order of magnitude of $c(n)$, but often we are interested in explicit upper bounds as well. In 1975 Hermann \cite{Her} established that
\begin{align*}
c(n) \leq 6 e^{4\pi \sqrt{n}},
\end{align*}
while the state-of-the-art result of Brisebarre and Philibert \cite{BP} yields the asymptotically sharp bound
\begin{align}\label{BP j bound}
c(n) = \frac{1}{\sqrt{2} n^{3/4}} e^{4\pi \sqrt{n}} \left(1 - \frac{3}{32\pi \sqrt{n}} + \varepsilon_n \right), \ |\varepsilon_n| \leq \frac{0.055}{n}.
\end{align}
The proof of (\ref{Jenkins Rouse Thm}) makes use of (\ref{BP j bound}) to bound the tails of certain infinite series. Similarly, our proof of Theorem \ref{Main Theorem} requires bounds on the coefficients of certain modular functions for $\Gamma_0(2)$. These modular functions are actually Hauptmoduln for $\Gamma_0(2)$ (see Section \ref{Background section}). We denote these modular functions by $\psi$ and $\phi$, defined as
\begin{align*}
\psi(z) &= \frac{\Delta(z)}{\Delta(2z)} = q^{-1} - 24 + \sum_{n=1}^\infty s(n) q^n, \\
\phi(z) &= \frac{1}{\psi(z)} = \sum_{n=1}^\infty b(n) q^n. \notag
\end{align*}
Here the $s(n)$ and $b(n)$ are integers, and one can show the $b(n)$ are positive. Our next theorem gives asymptotic formulas for $|s(n)|$ and $b(n)$ as $n \rightarrow \infty$, similar to (\ref{j coeff asymp}).
\begin{theorem}\label{asymp for Hauptmods}
Let $s(n)$ and $b(n)$ be given as above. Then
\begin{align*}
|s(n)| &\sim \frac{1}{2 n^{3/4}} e^{2\pi \sqrt{n}}, \\
b(n) &\sim \frac{2^{1/4}}{8192} \frac{1}{n^{3/4}} e^{2\pi \sqrt{2n}},
\end{align*}
as $n \rightarrow \infty$.
\end{theorem}
This theorem supports the general principle that the coefficients of modular functions are large. The proof is straightforward, relying on a result of Dewar and Murty \cite{DM}.

Our last theorem is an explicit upper bound on $|s(n)|$ and $b(n)$.
\begin{theorem}\label{bounds for Hauptmods}
Let $s(n)$ and $b(n)$ be given as above. Then
\begin{align*}
|s(n)| &< 0.9 \cdot n^{11} \cdot e^{2\pi \sqrt{2n}}, \\
b(n) &< 0.08 \cdot n^{11} \cdot e^{2\pi \sqrt{2n}}
\end{align*}
for $n \geq 1$.
\end{theorem}
By comparison with Theorem \ref{asymp for Hauptmods} we see these bounds may be improved, but the bounds suffice for our purposes. The proof of this theorem is elementary, using only an explicit bound on the number of partitions of an integer into distinct parts. Indeed, interpreted appropriately, the proof of Theorem \ref{bounds for Hauptmods} gives explicit upper bounds for $r$-colored partitions of $n$ into distinct parts, with $r=2,4,8,16,24$ (see \cite{EO} for definitions).

The outline of the rest of the paper is as follows. Section \ref{Background section} covers necessary background material about modular forms. In Section \ref{Hauptmod section} we prove Theorems \ref{asymp for Hauptmods} and \ref{bounds for Hauptmods}. In Section \ref{L-fn section} we prove some necessary lemmas about $L$-functions. We derive an upper bound for $\langle F_{k,m}, F_{k,m} \rangle$ in Section \ref{Petersson upp bound section}, and in Section \ref{proof of main thm section} we prove Theorem \ref{Main Theorem}.

\section{Background}\label{Background section}
Here we give some necessary background and definitions about modular forms (see e.g. section 2 of \cite{Rou2}). We let $\mathbb{H}$ denote the complex upper half-plane. For a positive integer $N \geq 1$ define
\begin{align*}
\Gamma_0(N) = \left\{ 
\begin{pmatrix}
a & b \\
c & d
\end{pmatrix}
\in \text{SL}_2(\mathbb{Z}) : c \equiv 0 \pmod{N}
\right\}.
\end{align*}
We let $S_k(N)$ denote the finite-dimensional $\mathbb{C}$-vector space of cusp forms of weight $k$ for $\Gamma_0(N)$.

If $f$ is a modular form of weight $k$ for some $\Gamma_0(N)$ and
\begin{align*}
\alpha = \begin{pmatrix}
a & b \\
c & d
\end{pmatrix}
\in \text{GL}_2(\mathbb{Q}), \  \text{det}(\alpha) > 0,
\end{align*}
we define the slash operator $f|_{\alpha}$ by
\begin{align*}
f|_{\alpha}  = (\text{det } \alpha)^{k/2} (cz+d)^{-k} f \left(\frac{az+b}{cz+d} \right).
\end{align*}
Here we follow the notation of \cite{Rou2}. Note that some authors replace the exponent $\frac{k}{2}$ by $k-1$. For a positive integer $d$ we also define the operator $V_d$ by
\begin{align*}
f(z) | V_d = f(dz).
\end{align*}
It is well-known that $V_d$ maps $S_k(N)$ to $S_k(dN)$. We further have the usual Hecke operator $T_p$, defined for $p \nmid N$ by
\begin{align*}
\sum_{n=0}^\infty a(n) q^n | T_p = \sum_{n=0}^\infty \left(a(pn) + p^{k-1} a \left(\frac{n}{p} \right) \right) q^n,
\end{align*}
where $a \left(\frac{n}{p} \right) = 0$ if $p \nmid n$. Hecke operators for $p | N$ are defined differently. The Hecke operators preserve $S_k(N)$.

If $f,g \in S_k(N)$, we define their Petersson inner product $\langle f,g \rangle$ by
\begin{align*}
\langle f,g \rangle &= \frac{3}{\pi [\text{SL}_2(\mathbb{Z}) : \Gamma_0(N)]} \int_{\mathbb{H} /  \Gamma_0(N)} f(x+iy) \overline{g(x+iy)} y^k \frac{dx dy}{y^2}.
\end{align*}
The integration takes place over a fundamental domain for $\Gamma_0(N)$, and the Petersson inner product is well-defined with respect to choice of fundamental domain (equivalently, a choice of coset representatives for $\Gamma_0(N)$ in $\text{SL}_2(\mathbb{Z})$). Additionally, if $\alpha \in \text{GL}_2(\mathbb{Q})$ with positive determinant, then $\langle f|_{\alpha}, g|_{\alpha} \rangle = \langle f,g \rangle$.

The space of oldforms of $S_k(N)$ is the space spanned by all forms
\begin{align*}
f(z) |V_d, \text{ where } f(z) \in S_k(M)
\end{align*}
and we have $M|N, M<N$ and $d$ is a divisor of $\frac{N}{M}$. We define $S_k^{\text{new}}(N)$ to be the orthogonal complement of the oldforms in $S_k(N)$ with respect to the Petersson inner product. A newform of level $N$ is a form
\begin{align*}
f(z) = \sum_{n=1}^\infty a(n) q^n \in S_k^{\text{new}}(N)
\end{align*}
that is a simultaneous eigenform of all the Hecke operators $T_p$, normalized with $a(1)=1$. It is a well-known property of newforms that if $f_1 \neq f_2$ are newforms, then
\begin{align*}
\langle f_1, f_2 \rangle = 0.
\end{align*}

Lastly we must define some modular forms. We let $\eta(z)$ be the usual Dedekind eta function
\begin{align*}
\eta(z) = q^{\frac{1}{24}} \prod_{n=1}^\infty (1-q^n).
\end{align*}
A \emph{Hauptmodul} (plural \emph{Hauptmoduln}) for a subgroup $\Gamma \subset \text{SL}_2(\mathbb{Z})$ is a modular function $f$ so that the field of all modular functions for $\Gamma$ is $\mathbb{C}(f)$. We define the Hauptmoduln $\psi,\phi$ for $\Gamma_0(2)$ by
\begin{align*}
\psi(z) &= \left(\frac{\eta(z)}{\eta(2z)} \right)^{24} = q^{-1} - 24 + \sum_{n=1}^\infty s(n) q^n, \\
\phi(z) &= \frac{1}{\psi(z)} = \sum_{n=1}^\infty b(n) q^n. \notag
\end{align*}
We note that $\psi(z)$ has a pole at infinity and vanishes at zero, while $\phi(z)$ has a pole at zero and vanishes at infinity.

Let $E_k$ denote the usual Eisenstein series of weight $k$ for $\text{SL}_2(\mathbb{Z})$. Thus
\begin{align*}
E_k(z) &= 1 - \frac{2k}{B_k} \sum_{n=1}^\infty \sigma_{k-1}(n) q^n,
\end{align*}
with $B_k$ denoting the $k$th Bernoulli number. We define modular forms $S_4$ and $F_2$ for $\Gamma_0(2)$ by
\begin{align*}
S_4(z) &= \frac{E_4(z) - E_4(2z)}{240}, \\
F_2(z) &= 2E_2(2z) - E_2(z). \notag
\end{align*}
Note that $S_4$ is modular for $\Gamma_0(2)$ of weight 4, and $F_2$ is modular for $\Gamma_0(2)$ of weight 2. The generating function for $F_{k,m}$ involves both $S_4$ and $F_2$, as well as $\psi$ (see (\ref{Gen Func for Fkm}) in Section \ref{Petersson upp bound section}).

\section{Hauptmodul Coefficients}\label{Hauptmod section}
In this section we prove asymptotics for the coefficients of $\psi$ and $\phi$, and also explicit upper bounds on the absolute values of their coefficients.

We begin by finding asymptotic formulas for the coefficients. The key ingredient is the following theorem due to Dewar and Murty (Theorem 2, \cite{DM}).
\begin{theorem}\label{DM asymp thm}
Suppose
\begin{align*}
f(z) &= \sum_{n=0}^\infty \lambda_f(n) q^n, \\
g(z) &= \sum_{n=0}^\infty \lambda_g(n) q^n,
\end{align*}
where
\begin{align*}
\lambda_f(n) &\sim c_fn^\alpha e^{A \sqrt{n}}, \\
\lambda_g(n) &\sim c_gn^\beta e^{B \sqrt{n}},
\end{align*}
with $\alpha,\beta,A,B,c_f,c_g \in \mathbb{R}$ and $A,B,c_f,c_g > 0$. Then for $fg(z) = \sum_{n=0}^\infty \lambda_{fg}(n) q^n$ we have
\begin{align*}
\lambda_{fg} \sim c_f c_g 2 \sqrt{2\pi} \frac{A^{2\alpha + 1} B^{2\beta+1}}{(A^2 + B^2)^{\frac{5}{4}+\alpha + \beta}} n^{\alpha + \beta + \frac{3}{4}} e^{\sqrt{A^2 + B^2} \sqrt{n}}.
\end{align*}
\end{theorem}
We first consider the coefficients $s(n)$ of $\psi$. To see that the $s(n)$ are all integers, note that
\begin{align*}
\psi(z) &= q^{-1} \prod_{n=1}^\infty \frac{(1-q^n)^{24}}{(1-q^{2n})^{24}} = q^{-1} \prod_{n=1}^\infty (1-q^{2n-1})^{24}.
\end{align*}
If we define $k(z)$ by
\begin{align*}
k(z) = \prod_{n=1}^\infty (1-q^{2n-1}) = 1+\sum_{n=1}^\infty g(n) q^n,
\end{align*}
then by examining $-q^{-1} k^{24}(z) = -\psi \left( z + \frac{1}{2} \right)$ in two different ways, we see that $(-1)^{n+1} s(n)$ is a positive integer.

In view of Theorem \ref{DM asymp thm} we must get an asymptotic formula for $g(n)$. It is not difficult to see that
\begin{align*}
g(n) &= (-1)^n \# \{\text{partitions of } n \text{ into distinct odd parts}\}.
\end{align*}
The function $g(n)$ is studied in \cite{Hag}, where it is shown that
\begin{align*}
|g(n)| \sim \frac{\sqrt{6}}{24^{3/4} n^{3/4}} e^{\pi \sqrt{\frac{n}{6}}}.
\end{align*}
The first part of Theorem \ref{asymp for Hauptmods} now follows by applying Theorem \ref{DM asymp thm} repeatedly (note that dividing by $q$ does not change the asymptotic).

Now we turn to the coefficients of $\phi$. We easily see from its definition that $\phi$ has positive integral coefficients, since
\begin{align*}
\phi(z) &= q\prod_{n=1}^\infty \frac{(1-q^{2n})^{24}}{(1-q^n)^{24}} = q\prod_{n=1}^\infty (1+q^n)^{24}.
\end{align*}
We recognize
\begin{align*}
\prod_{n=1}^\infty (1+q^n)
\end{align*}
as the generating function for $Q(n)$, the number of partitions of $n$ into distinct parts. An asymptotic formula for $Q(n)$ is given by
\begin{align*}
Q(n) &\sim \frac{1}{4 \cdot 3^{1/4} n^{3/4}} e^{\pi \sqrt{\frac{n}{3}}},
\end{align*}
as found, for example, in equation 4 of \cite{Bi}. Again applying Theorem \ref{DM asymp thm} repeatedly, we finish the proof of Theorem \ref{asymp for Hauptmods}. Using this same method one can establish by induction asymptotic formulas for coefficients of powers of $\psi$ and $\phi$.

It remains to prove Theorem \ref{bounds for Hauptmods}. Define coefficients $Q_k(n)$ by
\begin{align*}
\prod_{n=1}^\infty (1+q^n)^k = \sum_{n=0}^\infty Q_k(n) q^n,
\end{align*}
and note that $Q_k(0)=1$. Observe that $|g(n)|$ is less than the number of partitions of $n$ into distinct parts. Thus to get upper bounds on $|s(n)|$ and $b(n)$ it suffices to get an upper bound on $Q_{24}(n)$ and then determine the effect of dividing or multiplying by $q$.

Corollary 2 of \cite{Bi} shows that
\begin{align*}
Q(n) = Q_1(n) < \frac{\pi}{2\sqrt{3n}} \cdot e^{\pi \sqrt{\frac{n}{3}}},
\end{align*}
which implies
\begin{align*}
Q_2(n) &= \sum_{k=0}^n Q_1(k)Q_1(n-k) = 2Q_1(n) + \sum_{k=1}^{n-1} Q_1(k)Q_1(n-k) \\
&< \frac{\pi}{\sqrt{3n}}e^{\pi \sqrt{\frac{n}{3}}} + \frac{\pi^2}{12} \sum_{k=1}^{n-1} \frac{1}{\sqrt{kn-k^2}} e^{\pi(\sqrt{\frac{k}{3}} + \sqrt{\frac{n-k}{3}})} \notag \\
&< \frac{\pi}{\sqrt{3n}}e^{\pi \sqrt{\frac{n}{3}}} + \frac{\pi^2}{12} e^{\pi\sqrt{\frac{2n}{3}}}\sum_{k=1}^{n-1} \frac{1}{\sqrt{kn-k^2}}. \notag
\end{align*}
Consider the sum
\begin{align*}
\sum_{k \leq x} \frac{1}{\sqrt{kt-k^2}},
\end{align*}
with $x,t$ real numbers and $x<t$. Applying partial summation in the usual way and simplifying,
\begin{align*}
\sum_{k \leq x} \frac{1}{\sqrt{kt-k^2}} \leq 2\arctan\left( \frac{x^{1/2}}{(t-x)^{1/2}} \right) + \frac{1}{(t-1)^{1/2}} + \frac{2x^{1/2}}{t(t-x)^{1/2}}.
\end{align*}
Setting $x=n-1,t=n$, we get
\begin{align*}
\sum_{k=1}^{n-1} \frac{1}{\sqrt{kn-k^2}} \leq 2 \arctan\left( \frac{\sqrt{n-1}}{1}\right)+\frac{1}{\sqrt{n-1}} + \frac{2}{\sqrt{n}} \leq \pi + \frac{1}{\sqrt{n-1}} + \frac{2}{\sqrt{n}}.
\end{align*}
Thus
\begin{align*}
Q_2(n) &< \frac{\pi}{\sqrt{3n}}e^{\pi \sqrt{\frac{n}{3}}} + \frac{\pi^2}{12} e^{\pi\sqrt{\frac{2n}{3}}} \left(\pi + \frac{1}{\sqrt{n-1}} + \frac{2}{\sqrt{n}} \right) \\
&< 3.44 \cdot e^{\pi\sqrt{\frac{2n}{3}}} \notag
\end{align*}
if $n \geq 10$. We proceed similarly with $Q_4(n)$, obtaining
\begin{align*}
Q_4(n) &< 2 \cdot 3.44 e^{\pi \sqrt{\frac{2n}{3}}} + 3.44^2 \sum_{k=1}^{n-1}e^{\pi(\sqrt{\frac{2n}{3}}+\sqrt{\frac{2(n-k)}{3}})} \\
&< 12.08 \cdot n \cdot e^{2\pi \sqrt{\frac{n}{3}}} \notag.
\end{align*}
Continuing in this manner we get
\begin{align*}
Q_8(n) &< 24.33 \cdot n^3 \cdot e^{2\pi \sqrt{\frac{2n}{3}}}, \\
Q_{16}(n) &< 4.23 \cdot n^7 \cdot e^{4\pi \sqrt{\frac{n}{3}}}, \\
Q_{24}(n) &< 0.08 \cdot n^{11} \cdot e^{2\pi \sqrt{2n}}.
\end{align*}
Adjusting for multiplication or division by $q$ and doing a manual check for $1 \leq n \leq 10$, we finish the proof of Theorem \ref{bounds for Hauptmods}.

It is natural to ask whether our elementary methods give similar results for Hauptmoduln of higher levels (particularly levels 3, 5, and 7). Suitable modifications of our arguments will give explicit, but weak, upper bounds. Tighter bounds require bounds on restricted integer partitions. On the other hand, obtaining asymptotic formulas via our method seems difficult. For example, consider using our method to obtain an asymptotic formula for the coefficients of the level 3 analogue of $\psi$. We find that we would need an asymptotic formula for the absolute value of
\begin{align*}
\#&\left\{ \text{partitions of $n$ into even number of distinct parts congruent to 1 mod 3} \right\} \\
&-\#\left\{ \text{partitions of $n$ into odd number of distinct parts congruent to 1 mod 3} \right\}. \notag
\end{align*}
We would also need a similar asymptotic formula for parts that are 2 mod 3. As the level increases the restricted partitions become more complex, hence an extension of our method to higher levels appears nontrivial. The circle method or other more advanced techniques likely yield satisfactory results (e.g. compare (\ref{j coeff asymp}) and the main theorem of \cite{Rad}).

\section{$L$-function Calculations}\label{L-fn section}
The goal of this section is to derive a lower bound on the $L$-function special value $L(\text{Sym}^2 g,1)$ (see below for definitions), where $g \in S_k(2)$ is a newform. We proceed in slightly more generality, treating $g \in S_k^{\text{new}}(p)$, for $p \in \{2,3,5,7\}$. We are interested in the special value $L(\text{Sym}^2 g,1)$ because of the well-known identity
\begin{align*}
\langle g,g \rangle &= \frac{6}{\pi^2} \cdot \frac{1}{1 + \frac{1}{p}} \cdot \frac{\Gamma(k)}{(4\pi)^k} L(\text{Sym}^2 g,1).
\end{align*}

If $g \in S_k(p)$ has coefficients $\{a(n)\}_{n=1}^\infty$, then we define the normalized $L$-function of $g$ to be
\begin{align*}
L(g,s) = \prod_{q \text{ prime}} (1-\alpha_q q^{-s})^{-1}(1-\beta_q q^{-s})^{-1},
\end{align*}
where $\alpha_q + \beta_q = \frac{a(q)}{q^{\frac{k-1}{2}}}$ and $\alpha_q \beta_q = 1$ for $q \neq p$, and we allow $\alpha_p$ or $\beta_p$ to be zero. The symmetric square $L$-function $L(\text{Sym}^2 g,s)$ associated to $g$ is then given by
\begin{align*}
L(\text{Sym}^2 g,s) = \prod_{q \text{ prime}} (1-\alpha_q^2 q^{-s})^{-1}(1-q^{-s})^{-1}(1-\beta_q^2 q^{-s})^{-1};
\end{align*}
see Section 3 of \cite{CM} for the computation of these local factors. The symmetric square $L$-function is known by the work of Gelbart and Jacquet \cite{GeJa} to be the $L$-function of a cuspidal automorphic representation on $\text{GL}(3)$, hence is entire and has a functional equation of the usual form: if we set
\begin{align*}
\Lambda(\text{Sym}^2 g,s) &= p^s \pi^{-\frac{3s}{2}} \Gamma \left( \frac{s+1}{2} \right)\Gamma \left( \frac{s+k-1}{2} \right)\Gamma \left( \frac{s+k}{2} \right)L(\text{Sym}^2 g,s),
\end{align*}
then we have the functional equation
\begin{align*}
\Lambda(\text{Sym}^2 g,s) = \Lambda(\text{Sym}^2 g,1-s).
\end{align*}
Our first step toward obtaining a lower bound on $L(\text{Sym}^2 g,1)$ is to show that $L(\text{Sym}^2 g,s)$ has no Siegel zeros.
\begin{lemma}\label{ZeroFreeRegionLemma}
Let $g \in S_k^{\emph{new}}(p)$, with $p \in \{2,3,5,7\}$. Then
\begin{align*}
L(\emph{Sym}^2 g,s) \not = 0
\end{align*}
for real $s$ with $s > 1-\frac{5-2\sqrt{6}}{10 \log k}$.
\end{lemma}
\begin{proof}
Here we follow \cite{Rou1} and \cite{KR}, which are based on an argument of Goldfeld, Hoffstein, and Lieman \cite{HL}. 

Consider the function
\begin{align*}
L(s)=\zeta(s)^2L(\text{Sym}^2g,s)^3 L(\text{Sym}^4g,s),
\end{align*}
where $L(\text{Sym}^4g,s)$ is given by
\begin{align*}
L(\text{Sym}^4g,s) &= \prod_{q \text{ prime}} (1-\alpha_q^4 q^{-s})^{-1}(1-\alpha_q^2 q^{-s})^{-1}(1-q^{-s})^{-1} (1-\beta_q^2 q^{-s})^{-1} (1-\beta_q^4 q^{-s})^{-1}
\end{align*}
and $\alpha_q,\beta_q$ as above (as with the symmetric square $L$-function, see \cite{CM} for further details). Kim \cite{Kim} proved that $L(\text{Sym}^4 g,s)$ is associated to an automorphic representation on $\text{GL}(5)$, so $L(\text{Sym}^4 g,s)$ has an analytic continuation and functional equation of the usual type.

From the above remarks it follows that $L(s)$ has a functional equation and is analytic except for a pole at $s=1$. Define $\Lambda(s)=s^2(1-s)^2G(s)L(s)$, where
\begin{align*}
G(s)=&p^{5s}\pi^{-8s}\Gamma\left(\frac{s}{2}\right)^3 \Gamma\left(\frac{s+1}{2}\right)^3 \Gamma\left(\frac{s+k-1}{2}\right)^4 \\
&\times \Gamma\left(\frac{s+k}{2}\right)^4 \Gamma\left(\frac{s+2k-2}{2}\right) \Gamma\left(\frac{s+2k-1}{2}\right). \notag
\end{align*}
Then $\Lambda(s)=\Lambda(1-s)$. For the duration of the proof we take $s$ real and greater than 1. As $\Lambda(s)$ is an entire function of order 1, it admits a Hadamard product factorization
\begin{align*}
\Lambda(s)=e^{A+Bs}\prod_\rho \left(1-\frac{s}{\rho} \right)e^{s/\rho},
\end{align*}
where $A,B$ are constants and the product is over all zeros $\rho$ of $\Lambda(s)$. Taking the logarithmic derivative, we have
\begin{align*}
\sum_{\rho} \left(\frac{1}{s-\rho} + \frac{1}{\rho} \right)=\frac{2}{s}+\frac{2}{s-1}+\frac{G'(s)}{G(s)}+\frac{L'(s)}{L(s)}-B.
\end{align*}
Since $s>1$, we have $L(s)>0$ and $L'(s)<0$, and hence $\frac{L'(s)}{L(s)} < 0$. Now we take the real part of both sides and use the fact that $\text{Re}(B)=-\sum_{\rho} \text{Re}(\frac{1}{\rho})$ (Theorem 5.6 in \cite{IK}). Therefore
\begin{align*}
\sum_{\rho} \text{Re} \left(\frac{1}{s-\rho} \right) \leq \frac{2}{s}+\frac{2}{1-s} + \frac{G'(s)}{G(s)}.
\end{align*}
Define $\psi(s)=\frac{\Gamma'(s)}{\Gamma(s)}$ (the notational conflict with the Hauptmodul $\psi$ is unfortunate, but both notations are standard). Then
\begin{align*}
\frac{G'(s)}{G(s)}&=5\log(p)-8\log(\pi)+\frac{3}{2}\psi\left(\frac{s}{2}\right)+\frac{3}{2}\psi\left(\frac{s+1}{2}\right)+2\psi\left(\frac{s+k-1}{2}\right)  \\
&+ 2\psi\left(\frac{s+k-1}{2}\right) + \frac{1}{2}\psi\left(\frac{s+2k-2}{2}\right)+\frac{1}{2}\psi\left(\frac{s+2k-1}{2}\right). \nonumber
\end{align*}
Now set $s=1+\alpha$ with $0<\alpha \leq \frac{1}{2}$ to be chosen shortly. With $p \in \{2,3,5,7\}$ and the restriction on $s$, we have that $\frac{G'(s)}{G(s)} \leq 10\log(k)-2$, since $\Psi(s) \leq \log s$ for $s \geq 1$. Now assume $\beta$ is a real zero of $L(\text{Sym}^2g,s)$. Then
\begin{align*}
\frac{3}{1+\alpha-\beta} \leq 2 + \frac{2}{\alpha} + \frac{G'(s)}{G(s)} \leq \frac{2}{\alpha}+10\log(k).
\end{align*}
Choosing the optimum value of $\alpha=\frac{\sqrt{6}-2}{10\log(k)}$, which is always less than $1/2$, completes the proof.
\end{proof}
We use the above lemma to get a lower bound on $L(\text{Sym}^2 g,1)$, following Rouse \cite{Rou1} and an argument of Hoffstein \cite{Hof}.
\begin{lemma}\label{Newform Lower Bound}
Let $k$ be an even integer with $k \geq 8$, and let $p \in \{2,3,5,7\}$. If $g$ is a normalized newform in $S_k^{\text{new}}(p)$ then
\begin{align*}
L(\emph{Sym}^2g,1) > \frac{1}{86\log(k)}.
\end{align*}
\end{lemma}
\begin{proof}
We have the Rankin-Selberg convolution $L$-function
\begin{align*}
L(g \otimes g,s)=\zeta(s)L(\text{Sym}^2g,s)=\sum_{n=1}^\infty \frac{a(n)}{n^s}.
\end{align*}
Obviously $L(g \otimes g,s)$ has a functional equation since both $\zeta(s)$ and $L(\text{Sym}^2f,s)$ do. By checking Euler factors one may show that $a(n)\geq 0$ and $a(n^2) \geq 1$ for all $n$. Let $\beta=1-\frac{5-2\sqrt{6}}{10 \log k}$, and note that $\frac{9}{10}<\beta<1$. Set $x=k^A$, where $A$ is a parameter to be chosen at the end of the proof (in the end we will set $A = \frac{16}{5}$). Consider the integral
\begin{align*}
I=\frac{1}{2\pi i} \int_{2-i\infty}^{2+i \infty} \frac{L(g \otimes g,s + \beta)x^s}{s \prod_{r=2}^{10}(s+r)}ds.
\end{align*}
We use the fact that
\begin{align*}
\frac{1}{2\pi i} \int_{2-i\infty}^{2+i \infty} \frac{x^s}{s \prod_{r=2}^{10}(s+r)}ds = 
\begin{cases}
\frac{(x+9)(x-1)^9}{10!x^{10}}, \text{ if } x>1 \\
0, \text{ if } x<1,
\end{cases}
\end{align*}
which implies
\begin{align*}
I=\sum_{n \leq x} \frac{a(n)(\frac{x}{n}+9)(\frac{x}{n}-1)^9}{10! n^{\beta}(\frac{x}{n})^{10}}.
\end{align*}
We take only those terms with $\frac{x}{n}\geq 48$, giving
\begin{align*}
I \geq \frac{1}{10!}\frac{(48+9)(48-1)^9}{48^{10}}\sum_{n \leq \sqrt{\frac{x}{48}}} \frac{1}{n^2} \geq \frac{1.39873}{10!}.
\end{align*}
Let $\alpha=-\frac{3}{2}-\beta$. We shift the contour in $I$ to $\alpha$, and thereby pick up poles at $s=1-\beta$, $s=0$, and $s=-2$. By the Residue Theorem, we have
\begin{align*}
I=&\frac{1}{2\pi i}\int_{\alpha-i\infty}^{\alpha+i \infty} \frac{L(g \otimes g,s + \beta)x^s}{s \prod_{r=2}^{10}(s+r)}ds + \frac{L(\text{Sym}^2g,1)x^{1-\beta}}{(1-\beta)\prod_{r=2}^{10}(1-\beta+r)} \\
&+\frac{L(g \otimes g,\beta)}{10!}-\frac{L(g \otimes g,-2+\beta)x^{-2}}{2 \cdot 8!}. \notag
\end{align*}
Lemma \ref{ZeroFreeRegionLemma} shows that $L(\text{Sym}^2g,s)$ has no real zeros to the right of $\beta$, so $L(\text{Sym}^2g,\beta)\geq 0$. As $\zeta(\beta)<0$, we see that $L(g \otimes g,\beta)\leq 0$. Similarly, we have that $L(\text{Sym}^2g,-2+\beta)<0$ and $\zeta(-2+\beta)<0$, so $L(g \otimes g,-2+\beta) >0$. Thus, we have
\begin{align}\label{CentralLowerValueInequality}
I-I_2 \leq \frac{L(\text{Sym}^2g,1)x^{1-\beta}}{(1-\beta)\prod_{r=2}^{10}(1-\beta+r)},
\end{align}
where we have defined
\begin{align*}
I_2=\frac{1}{2\pi i}\int_{\alpha-i\infty}^{\alpha+i \infty} \frac{L(g \otimes g,s + \beta)x^s}{s \prod_{r=2}^{10}(s+r)}ds.
\end{align*}
We require an upper bound $|I_2|$. Using the functional equations for $L(g \otimes g,s)$ and $\Gamma(s)$ we have
\begin{align*}
\left|L\left(g \otimes g,-\frac{3}{2}+it\right)\right| &= p^4\pi^{-8} \left| \frac{1}{4}+ \frac{it}{2} \right|^2 \left| \frac{3}{4}+ \frac{it}{2} \right|^2 \left| \frac{k}{2} - \frac{5}{4} - \frac{it}{2} \right|
\left| \frac{k}{2} - \frac{3}{4} - \frac{it}{2} \right| \\
&\times \left| \frac{k}{2} - \frac{1}{4} - \frac{it}{2} \right| \left| \frac{k}{2} + \frac{1}{4} - \frac{it}{2} \right|\left|L\left(f \otimes f,\frac{5}{2}-it\right)\right|. \nonumber
\end{align*}
Recall that $|x^s|=k^{-A(\frac{3}{2}+\beta)}$. We have the bounds
\begin{align*}
\left|L(g \otimes g,\frac{5}{2}-it)\right| \leq \zeta \left(\frac{5}{2} \right)^4
\end{align*}
and
\begin{align*}
\frac{1}{|-3/2-\beta+it|\prod_{r=2}^{10}|r-3/2-\beta+it|} \leq \frac{1}{|12/5+it||2/5+it|\prod_{r=3}^{10}|r-5/2+it|}.
\end{align*}
Together this gives
\begin{align*}
|I_2| &\leq \frac{\zeta(5/2)^4 p^4 k^{4-A(3/2+\beta)}}{2^9 \pi^9} \cdot \int_{-\infty}^\infty \frac{|1/2+it|^2|3/2+it|^2|1+it|^3|256/225+it|}{|12/5+it||2/5+it|\prod_{r=3}^{10}|r-5/2+it|}dt \\
&\leq \frac{.18047}{10!}p^4k^{4-A(3/2+\beta)} \leq \frac{.18047}{10!}k^{8-A(3/2+\beta)}. \nonumber
\end{align*}
By (\ref{CentralLowerValueInequality}), this gives
\begin{align*}
L(\text{Sym}^2g,1)\geq (1-\beta)(1.39873k^{A(\beta-1)}-.18047k^{8-5A/2}).
\end{align*}
Choosing $A=\frac{16}{5}$ completes the proof.
\end{proof}

\section{Petersson Norm Upper Bounds}\label{Petersson upp bound section}
Recall the notation from the introduction. We write $k = 4\ell + k'$ where $k' \in \{0,2\}$. The dimension of $S_k(2)$ is $\ell - 1$. There is a basis for $S_k(2)$ indexed by $m$ given by
\begin{align*}
F_{k,m}(z) &= q^m + \sum_{n = \ell}^\infty A_k(m,n) q^n, \ 1 \leq m \leq \ell - 1.
\end{align*}
In this section we obtain an upper bound on $\langle F_{k,m}, F_{k,m} \rangle$. The $F_{k,m}$ have the following generating function (Section 6 of \cite{GJ}):
\begin{align*}
\sum_{m=-\infty}^{\ell - 1} F_{k,m}(z)e^{2\pi i m \tau}=\frac{(S_4^{\ell} \psi F_{k'})(z)}{(S_4^{\ell} \psi F_{k'})(\tau)} \frac{\psi(\tau)F_2(\tau)}{\psi(\tau)-\psi(z)}.
\end{align*}
Recall that $S_4$ and $F_2$ were defined in Section \ref{Background section}. We set $F_0(z) = 1$. Integrating the generating function gives an integral representation of $F_{k,m}(z)$, namely
\begin{align}\label{Gen Func for Fkm}
F_{k,m}(z) = \int_{-1/2}^{1/2} \frac{(S_4^\ell \psi F_{k'})(z)}{(S_4^\ell \psi F_{k'})(\tau)} \frac{\psi(\tau)F_2(\tau)}{\psi(\tau)-\psi(z)}e^{2\pi i m \tau} du,
\end{align}
where $\tau = u+iv$ and $v$ is a fixed constant to be chosen later. Here we are not free to take $v$ arbitrarily. We require $v$ to be positive and large enough such that it corresponds (under a change of variables) to a circle around $q=0$ that avoids poles in the integrand of (\ref{Gen Func for Fkm}); see the beginning of Section 3 in \cite{JR} for similar discussion.

To compute the inner product $\langle F_{k,m}, F_{k,m} \rangle$, we first need coset representatives of $\Gamma_0(2)$ in $\text{SL}_2(\mathbb{Z})$. The index of $\Gamma_0(2)$ in $\text{SL}_2(\mathbb{Z})$ is 3, and we choose the coset representatives
\begin{align*}
\alpha_1 &= 
\begin{pmatrix}
1 & 0 \\
0 & 1
\end{pmatrix}, \ \
\alpha_2 = 
\begin{pmatrix}
0 & -1 \\
1 & 0
\end{pmatrix}, \ \
\alpha_3 = 
\begin{pmatrix}
1 & 0 \\
1 & 1
\end{pmatrix}.
\end{align*}
Recall that the Petersson inner product is well-defined with respect to choice of coset representatives. If $\mathcal{F}$ denotes the usual fundamental domain for $\text{SL}_2(\mathbb{Z})$, then
\begin{align*}
\langle F_{k,m}, F_{k,m} \rangle &= \frac{1}{\pi} \int_{\mathbb{H}/\Gamma_0(2)} |F_{k,m}(x+iy)|^2 y^k \frac{dx dy}{y^2} \\
&= \frac{1}{\pi} \sum_{i=1}^3 \int_{\mathcal{F}} \left|F_{k,m}|_{\alpha_i^{-1}}(x+iy) \right|^2 y^{k-2} dx dy. \notag
\end{align*}
(The presence of $y^{k-2}$ in each integral is to ensure certain invariance properties are satisfied.) To get an upper bound on $\langle F_{k,m}, F_{k,m} \rangle$ we study $F_{k,m}|_{\alpha_i^{-1}}$. In particular, the Fourier expansion of $F_{k,m}$ changes when $F_{k,m}$ is acted on by these matrices, and we require upper bounds on the absolute values of these coefficients. The generating function representation of $F_{k,m}$ is key to this step. From the definition of the slash operator and the theory of Fourier expansions of modular forms, we have
\begin{align*}
F_{k,m}|_{\alpha_1^{-1}}(z) &= F_{k,m}(z) = q^m + \sum_{n=\ell}^\infty A_k^{(1)}(m,n)q^n, \\
F_{k,m}|_{\alpha_2^{-1}}(z) &= z^{-k} F_{k,m} \left(-\frac{1}{z} \right) = z^{-k} \sum_{n=1}^\infty A_k^{(2)}(m,n)q^{n/2}, \notag \\
F_{k,m}|_{\alpha_3^{-1}}(z) &= (-z+1)^{-k} F_{k,m} \left(\frac{z}{-z+1} \right) = (-z+1)^{-k} \sum_{n=1}^\infty A_k^{(3)}(m,n)q^{n/2}. \notag
\end{align*}
We write $A_k^{(i)}(m,n)$ to denote the dependence on $\alpha_i$. Note that $A_k^{(1)}(m,n) = A_k(m,n)$.

To get upper bounds on $|A_k^{(i)}(m,n)|$ we determine how replacing $z$ by $-\frac{1}{z}$ or $\frac{z}{-z+1}$ changes the integral representation of $F_{k,m}$. First consider $z \rightarrow - \frac{1}{z}$. Recall that $\eta(z)$ satisfies
\begin{align*}
\eta\left( - \frac{1}{z} \right) = \sqrt{-iz}\eta(z).
\end{align*}
Since $\psi$ is an eta quotient this implies
\begin{align*}
\psi\left( -\frac{1}{z} \right) = 2^{12} \phi \left(\frac{z}{2} \right).
\end{align*}
Transformation properties of Eisenstein series easily imply that
\begin{align*}
S_4\left(-\frac{1}{z} \right) &= \frac{z^4}{240} \left( E_4(z)-\frac{1}{16}E_4(z/2) \right), \\
F_2\left( -\frac{1}{z} \right) &= -\frac{z^2}{2} F_2 \left( \frac{z}{2} \right). \notag
\end{align*}
Next consider $z \rightarrow \frac{z}{-z+1}$. To compute $\eta \left(\frac{z}{-z+1} \right)$ we use, in addition to the transformation law above, the fact that
\begin{align*}
\eta \left( z + \frac{1}{2} \right) = \frac{e^{2\pi i/48}\eta^3(2z)}{\eta(z)\eta(4z)}.
\end{align*}
Standard computations then show that
\begin{align*}
\psi\left( \frac{z}{-z+1} \right) = -2^{12}\phi(z)\psi \left( \frac{z}{2} \right).
\end{align*}
By definition, 
\begin{align*}
S_4 \left(\frac{z}{-z+1} \right) = \frac{1}{240} \left(E_4 \left(\frac{z}{-z+1} \right)-E_4\left(\frac{2z}{-z+1}\right) \right),
\end{align*}
and a straightforward calculation shows that
\begin{align*}
E_4 \left(\frac{z}{-z+1} \right) = (z-1)^4 E_4(z).
\end{align*}
Observe that
\begin{align*}
E_4\left(\frac{2z}{-z+1}\right) &= E_4 \left(- \frac{1}{\frac{z-1}{2z}} \right) = \left( \frac{z-1}{2z} \right)^4 E_4 \left( \frac{z-1}{2z} \right).
\end{align*}
We see that $\frac{z-1}{2z}$ corresponds to the matrix $\begin{pmatrix}
1 & -1 \\
2 & 0
\end{pmatrix}$
acting on $z$. Certainly
\begin{align*}
\begin{pmatrix}
1 & -1 \\
2 & 0
\end{pmatrix}
=
\begin{pmatrix}
1 & 1 \\
2 & 3
\end{pmatrix}
\begin{pmatrix}
1 & -3 \\
0 & 2
\end{pmatrix},
\end{align*}
so if we set $w = \frac{z-3}{2}$ we have
\begin{align*}
E_4 \left( \frac{z-1}{2z} \right) &= E_4 \left(\frac{w+1}{2w+3} \right) = (2w+3)^4 E_4(w) \\
&= z^4 E_4 \left(\frac{z-3}{2} \right) = z^4 E_4 \left(\frac{z}{2} + \frac{1}{2} \right), \notag
\end{align*}
since
\begin{align*}
\begin{pmatrix}
1 & 1 \\
2 & 3
\end{pmatrix} \in \text{SL}_2(\mathbb{Z}).
\end{align*}
Putting this all together gives
\begin{align*}
S_4 \left( \frac{z}{-z+1} \right) = \frac{(z-1)^4}{240} \left(E_4(z) - \frac{1}{16} E_4 \left(\frac{z}{2} + \frac{1}{2} \right) \right).
\end{align*}

We do a similar calculation with $F_2 \left( \frac{z}{-z+1} \right)$ and find that
\begin{align*}
F_2 \left( \frac{z}{-z+1} \right) &= (z-1)^2 \left(\frac{1}{2}E_2\left(\frac{z}{2} + \frac{1}{2} \right)-E_2(z) \right).
\end{align*}
Putting everything together yields the required generating functions:
\begin{align*}
F_{k,m}(z) &= \int_{-1/2}^{1/2} \frac{\psi(\tau) F_2(\tau)}{(S_4^\ell \psi F_{k'})(\tau)} e^{2\pi i m \tau} \cdot \frac{ S_4^\ell(z) \psi(z) F_2(z)}{\psi(\tau)-\psi(z)} du, \\
F_{k,m}\left( -\frac{1}{z} \right) &= \int_{-1/2}^{1/2} \frac{\psi(\tau) F_2(\tau)}{(S_4^\ell \psi F_{k'})(\tau)} e^{2\pi i m \tau} \cdot \left(\frac{z^4}{240} \left(E_4(z) - \frac{1}{16}E_4 \left(\frac{z}{2} \right) \right) \right)^\ell \notag \\
&\ \ \ \cdot \frac{2^{12}\phi \left(\frac{z}{2} \right) \cdot \frac{-z^2}{2} F_2 \left(\frac{z}{2} \right)}{\psi(\tau)-2^{12}\phi\left(\frac{z}{2} \right)} du, \notag \\
F_{k,m} \left(\frac{z}{-z+1} \right) &= \int_{-1/2}^{1/2} \frac{\psi(\tau) F_2(\tau)}{(S_4^\ell \psi F_{k'})(\tau)} e^{2\pi i m \tau} \cdot \left(\frac{(z-1)^4}{240} \left(E_4(z) - \frac{1}{16} E_4 \left(\frac{z}{2} + \frac{1}{2} \right) \right) \right)^\ell \notag \\
&\ \ \ \cdot \frac{-2^{12}\phi(z) \psi \left(\frac{z}{2} \right) \cdot (z-1)^2 \left(\frac{1}{2}E_2\left(\frac{z}{2} + \frac{1}{2} \right)-E_2(z) \right)}{\psi(\tau)+2^{12}\phi(z) \psi \left(\frac{z}{2} \right)}du \notag.
\end{align*}
We write $z=x+iy$, where we take $y$ a constant to be fixed soon. The coefficients $A_k^{(i)}(m,n)$ are then given by
\begin{align*}
A_{k}^{(1)}(m,n) &= \int_{-1/2}^{1/2} F_{k,m}(z) e^{-2\pi i n z}dx, \\
A_{k}^{(2)}(m,n) &= \int_{-1}^1 F_{k,m} \left(-\frac{1}{z} \right) e^{-\pi i n z}dx, \notag \\
A_{k}^{(3)}(m,n) &= \int_{-1}^1 F_{k,m} \left(\frac{z}{-z+1} \right) e^{-\pi i n z}dx. \notag
\end{align*}
To get upper bounds on $|A_k^{(i)}(m,n)|$ it suffices to bound the appropriate double integrals. For the remainder of the section, we set $v = 1.16$ and $y = .865$, so that $\tau = u+1.16i, z = x+.865i$, where $u,x \in [-1/2,1/2]$. This choice of $v,y$ is identical to that in \cite{JR}. These values of $v$ and $y$ give reasonable bounds, and keep the difference of Hauptmoduln in the denominator of (\ref{Gen Func for Fkm}) far enough from zero. Further, in bounding $\langle F_{k,m},F_{k,m} \rangle$ we require certain infinite series to be convergent, and choosing $y$ to be just less than $\frac{\sqrt{3}}{2}$ makes this possible.

We handle each $A_k^{(i)}(m,n)$ in turn. For a modular form $G = \sum_{n=0}^\infty a(n)q^n$, we write $\tilde{G} = \sum_{n=0}^N a(n)q^n$, and write $RG = G-\tilde{G}$ (we also refer to $RG$ as the tail of the series). Throughout the calculations we understand $N$ to be equal to 100. All computations were performed using SAGE \cite{SAGE} and Mathematica \cite{Math}. (See Section 5 of \cite{GJ} for computations of a similar nature.)

\subsection{Upper Bound on $|A_k^{(1)}(m,n)|$} We have
\begin{align*}
A_{k}^{(1)}(m,n) &= \int_{-1/2}^{1/2} \int_{-1/2}^{1/2} \frac{\psi(\tau) F_2(\tau)}{(S_4^\ell \psi F_{k'})(\tau)} \cdot \frac{ S_4^\ell(z) \psi(z) F_2(z)}{\psi(\tau)-\psi(z)} e^{2\pi i m \tau} e^{-2\pi i n z}  du dx.
\end{align*}
Trivially bounding, we have
\begin{align*}
| A_{k}^{(1)}(m,n) | &\leq e^{-2\pi m \cdot 1.16} e^{2\pi n \cdot .865} \ \text{max}_{|u|,|x| \leq \frac{1}{2}} \left|\frac{F_2(\tau)}{(S_4^\ell F_{k'})(\tau)} \right| \left| \frac{ S_4^\ell(z) \psi(z) F_2(z)}{\psi(\tau)-\psi(z)} \right|.
\end{align*}
We find the appropriate upper and lower bounds on each piece separately, and then combine the bounds after all the calculations.

The Fourier expansion of $F_2(z)$ is given by
\begin{align*}
F_2(z) = 1+24 \sum_{n=1}^\infty \left( \sum_{\substack{d|n \\ d \text{ odd}}}d \right) q^n.
\end{align*}
It is clear that we have
\begin{align*}
|\tilde{F_2}(z)| &\leq 1 + 24\sum_{n=1}^N \left( \sum_{\substack{d|n \\ d \text{ odd}}}d \right) e^{-2\pi n \cdot .865} \leq 1.10514, \\
|\tilde{F_2}(\tau)| &\leq 1 + 24\sum_{n=1}^N \left( \sum_{\substack{d|n \\ d \text{ odd}}}d \right) e^{-2\pi n \cdot 1.16} \leq 1.01642. \notag
\end{align*}
To bound $RF_2$, we have
\begin{align*}
\sum_{\substack{d|n \\ d \text{ odd}}} d \leq \sigma(n) < n+n^2,
\end{align*}
and thus
\begin{align*}
|RF_2(\tau)|,|RF_2(z)| \leq \sum_{n=N+1}^\infty (n+n^2)e^{-2\pi n \cdot .865} \leq 4.15 \cdot 10^{-235},
\end{align*}
where we have used the technique of taking derivatives of geometric series to evaluate the sum. (Hereafter we do not mention the contribution of tails of series, since our choice of $N$ lets us easily show the tails are always negligible). Hence
\begin{align*}
|F_2(z)| &\leq 1.10514, \\
|F_2(\tau)| &\leq 1.01642. \notag
\end{align*}
We similarly find that
\begin{align*}
|S_4(z)| \leq 0.00452.
\end{align*}
The next task is getting a lower bound on $|S_4(\tau)|$. First, we take the derivative of $S_4(\tau)$ to get
\begin{align*}
\frac{d}{d\tau}S_4(\tau) = 2 \pi i \sum_{n=1}^\infty n \left(\sigma_3(n) - \sigma_3 \left( \frac{n}{2} \right) \right) e^{2\pi i n \tau},
\end{align*}
from which it follows that
\begin{align*}
\left|\frac{d}{d\tau}S_4(\tau) \right| \leq 2 \pi \sum_{n=1}^\infty (n^3+n^5)e^{-2\pi n \cdot 1.16} \leq 0.00871.
\end{align*}
We evaluate $\tilde{S_4}(\tau)$ at the points $\tau = \frac{n}{20000} + 1.16i$ for $n \in [-10000,10000]$ and find that the minimum absolute value at these points is greater than $0.000679.$ Hence the smallest possible value of $|\tilde{S_4}(\tau)|$ is given by
\begin{align*}
0.000679 - \frac{0.00871}{40000} \geq .00067.
\end{align*}
Subtracting off the tail, we see that
\begin{align*}
|S_4(\tau)| \geq .00067.
\end{align*}
It remains to handle $\psi(z)$ and $\psi(\tau)$. Taking the derivative, we have
\begin{align*}
\frac{d}{dz} \psi(z) = -2\pi i q^{-1} + 2\pi i \sum_{n=1}^\infty n s(n)q^n,
\end{align*}
and from this we get
\begin{align*}
\left|\frac{d}{dz} \psi(z) \right| &\leq 2\pi \left(e^{2\pi \cdot .865} + \sum_{n=1}^N n|s(n)|e^{-2\pi n \cdot .865} \right) + 2\pi \sum_{n=N+1}^\infty n|s(n)| e^{-2\pi n \cdot .865}.
\end{align*}
The first sum is finite, and we can evaluate it directly, getting
\begin{align*}
2\pi \left(e^{2\pi \cdot .865} + \sum_{n=1}^N n|s(n)|e^{-2\pi n \cdot .865} \right) \leq 1448.69599.
\end{align*}
For the infinite sum, we can bound $|s(n)|$ using Theorem \ref{bounds for Hauptmods}. Thus we have
\begin{align*}
2\pi \sum_{n=N+1}^\infty n|s(n)| e^{-2\pi n \cdot .865} &\leq 2\pi \cdot 0.9 \sum_{n=1}^\infty n^{12} e^{2\pi \sqrt{2n}-2\pi n \cdot .865}.
\end{align*}
We bound the exponent by
\begin{align*}
2\pi \sqrt{2n}-2\pi n \cdot .865 &= 2\pi n \left(-.865 + \sqrt{\frac{2}{n}} \right) \leq 2\pi n \left(-.865 + \sqrt{\frac{2}{N+1}} \right) \\
&\leq -2\pi n \cdot .724, \notag
\end{align*}
and summing the series we find that as usual the tail is negligible. It follows that
\begin{align*}
\left|\frac{d}{dz} \psi(z) \right| \leq 1448.69599.
\end{align*}

As above, we now calculate $|\tilde{\psi}(z)|$ on a grid of points. We look at points with $x = -\frac{1}{2} + \frac{n}{20000}, n \in [-20000,20000]$. We find that that maximum value on this point sample is given by $254.52626$, and thus we have
\begin{align*}
|\tilde{\psi}(z)| \leq 254.52626 + \frac{1448.69599}{40000} \leq 254.56248.
\end{align*}
This gives
\begin{align*}
|\psi(z)| \leq 254.56248.
\end{align*}
We similarly compute a lower bound for $|\psi(\tau)|$, finding that
\begin{align*}
|\psi(\tau)| \geq 1439.51688.
\end{align*}

We now pull all the computations above together to get
\begin{align}\label{Coeff Bound for A1}
|A_{k}^{(1)}(m,n)| &\leq e^{-2\pi m \cdot 1.16} e^{2\pi n \cdot .865} \ \text{max}_{|u|,|x| \leq \frac{1}{2}} \left|\frac{F_2(\tau)}{(S_4^\ell F_{k'})(\tau)} \right| \left| \frac{ S_4^\ell(z) \psi(z) F_2(z)}{\psi(\tau)-\psi(z)} \right| \\
&\leq e^{-2\pi m \cdot 1.16} e^{2\pi n \cdot .865} \cdot \frac{1.01642}{(.00067)^\ell} \cdot \frac{(.00452)^\ell (254.56248) (1.10514)}{1439.51688-254.56248} \notag \\
&\leq 0.242 \cdot 6.747^\ell \cdot e^{-2\pi m \cdot 1.16} e^{2\pi n \cdot .865}. \notag
\end{align}

\subsection{Upper Bound on $|A_k^{(2)}(m,n)|$} Here we have
\begin{align*}
|A_k^{(2)}(m,n)| &\leq 2^{13} e^{-2\pi m \cdot 1.16} e^{\pi n \cdot .865} \ \text{max}_{|u|,|x| \leq \frac{1}{2}} \left|\frac{F_2(\tau)}{(S_4^\ell F_{k'})(\tau)} \right| \frac{\left|\phi \left(\frac{z}{2} \right) \right| \cdot \left|\frac{z^2}{2} F_2 \left(\frac{z}{2} \right)\right|}{\left|\psi(\tau)-2^{12}\phi\left(\frac{z}{2} \right) \right|} \\
&\ \ \ \cdot \left|\frac{z^4}{240} \left(E_4(z) - \frac{1}{16}E_4 \left(\frac{z}{2} \right) \right) \right|^\ell \notag.
\end{align*}
We have already bounded everything involving $\tau$, so we only need to bound the parts involving $z$. The bound on $|z|$ is straightforward, since
\begin{align*}
|z| \leq |1/2+.865i| \leq .99912.
\end{align*}
Using this and arguing as above we find that
\begin{align*}
&\left|\frac{z^2}{2} F_2 \left( \frac{z}{2} \right) \right| \leq 1.35659, \\
&\left|\frac{z^4}{240} \left(E_4(z) - \frac{1}{16}E_4 \left(\frac{z}{2} \right) \right) \right| \leq .0042.
\end{align*}
Getting an upper bound on $|\phi(\frac{z}{2})|$ is entirely analogous to the calculations we have done before. We derive the bound
\begin{align*}
\left| \phi \left( \frac{z}{2} \right) \right| \leq 0.34276.
\end{align*}
Combining everything from above, we have
\begin{align}\label{Coeff Bound for A2}
|A_k^{(2)}(m,n)| &\leq e^{-2\pi m \cdot 1.16} e^{\pi n .865} \cdot 2^{13} \frac{1.01642}{(.00067)^\ell} \cdot \frac{(0.34276)(1.35659)(0.0042)^\ell}{1439.51688-2^{12} \cdot 0.34276} \\
&\leq 108.842 \cdot 6.269^\ell \cdot e^{-2\pi m \cdot 1.16} e^{\pi n \cdot .865}. \notag
\end{align}

\subsection{Upper Bound on $|A_k^{(3)}(m,n)|$} From equations above we have that
\begin{align*}
|A_k^{(3)}(m,n)| &\leq 2 e^{-2\pi m \cdot 1.16} e^{\pi n \cdot .865} \ \text{max}_{|x|,|u|\leq \frac{1}{2}} \left|\frac{F_2(\tau)}{(S_4^\ell F_{k'})(\tau)} \right| \left|\frac{(z-1)^4}{240} \left(E_4(z) - \frac{1}{16} E_4 \left(\frac{z}{2} + \frac{1}{2} \right) \right) \right|^\ell \\
&\ \ \ \cdot \frac{2^{12} \left|\phi(z) \psi \left(\frac{z}{2} \right) \right| \left|(z-1)^2 \left(\frac{1}{2}E_2\left(\frac{z}{2} + \frac{1}{2} \right)-E_2(z) \right) \right|}{\left|\psi(\tau)+2^{12}\phi(z) \psi \left(\frac{z}{2} \right) \right|} \notag.
\end{align*}
We first get a bound for $\frac{(z-1)^4}{240} \left(E_4(z) - \frac{1}{16} E_4 \left(\frac{z}{2} + \frac{1}{2} \right) \right)$. As usual, we bound its derivative, compute values of the truncated sum on a grid of points, and compensate for the effects of the derivative and the tail of the series. This gives
\begin{align*}
\left|\frac{(z-1)^4}{240} \left(E_4(z) - \frac{1}{16} E_4 \left(\frac{z}{2} + \frac{1}{2} \right) \right) \right| \leq 0.044063.
\end{align*}

Arguing similarly for $(z-1)^2 \left(\frac{1}{2}E_2 \left(\frac{z}{2} + \frac{1}{2} \right) -E_2(z)\right)$, we find that
\begin{align*}
\left|(z-1)^2 \left(\frac{1}{2}E_2 \left(\frac{z}{2} + \frac{1}{2} \right) -E_2(z)\right) \right| \leq 2.69392.
\end{align*}
It remains to deal with $\phi(z)$ and $\psi \left( \frac{z}{2} \right)$. Again the computations are routine, and we have
\begin{align*}
\left| \phi(z) \right| &\leq 0.00486, \\
\left| \psi \left( \frac{z}{2} \right) \right| &\leq 15.95619. \notag
\end{align*}

From our bounds above, we have
\begin{align}\label{Coeff Bound for A3}
|A_k^{(3)}(m,n)| &\leq e^{-2\pi m \cdot 1.16} e^{\pi n .865} 2^{13} \frac{1.01642}{.00067^\ell} \cdot \frac{.044063^\ell (2.69392)(.00486 \cdot 15.95619)}{1439.51688-2^{12} \cdot .00486 \cdot 15.95619} \\
&\leq 1.551 \cdot 65.766^\ell \cdot e^{-2\pi m \cdot 1.16} e^{\pi n .865}. \notag
\end{align}

Now we are in a position to get an upper bound on $\langle F_{k,m}, F_{k,m} \rangle$. We have
\begin{align*}
\langle F_{k,m},F_{k,m} \rangle = \frac{1}{\pi}\sum_{i=1}^3 \int_{\mathcal{F}} F_{k,m}|_{\alpha_i^{-1}} (z) \overline{F_{k,m}|_{\alpha_i^{-1}} (z)} y^{k-2} dx dy = \frac{1}{\pi} \left( I_1 + I_2 + I_3 \right). \notag
\end{align*}
We bound each of $I_1,I_2,I_3$ separately.

\subsection{Upper Bound on $I_1$}
In this integral we act on $F_{k,m}$ with $\alpha_1^{-1}$, which is just the identity matrix. We write
\begin{align*}
F_{k,m}(z) = q^m + \sum_{n=\ell}^\infty A_k^{(1)}(m,n) q^n,
\end{align*}
and putting this in the integral and integrating over $x$, we get
\begin{align*}
I_1 &= \int_{-1/2}^{1/2} \int_{\sqrt{1-x^2}}^\infty |F_{k,m}(x+iy)|^2 y^{k-2} dy dx \leq \int_{\sqrt{3}/2}^\infty y^{k-2} \int_{-1/2}^{1/2} |F_{k,m}(x+iy)|^2 dx dy \\ 
&= \int_{\sqrt{3}/2}^\infty y^{k-2} \left(e^{-4\pi m y} + \sum_{n=\ell}^\infty |A_k^{(1)}(m,n)|^2 e^{-4\pi n y} \right)dy \notag.
\end{align*}
Changing variables, we find that
\begin{align*}
I_1 &\leq \frac{(k-2)!}{(4\pi m)^{k-1}} + \frac{1}{(4\pi)^{k-1}} \sum_{n=\ell}^\infty \frac{|A_k^{(1)}(m,n)|^2}{n^{k-1}} \int_{2\pi \sqrt{3}n}^\infty u^{k-2}e^{-u}du.
\end{align*}
Thus
\begin{align*}
\int_{2\pi \sqrt{3}n}^\infty u^{k-2}e^{-u}du = e^{-2\pi \sqrt{3}n} \sum_{i=0}^{k-2} \frac{(k-2)!}{i!}(2\pi \sqrt{3}n)^i,
\end{align*}
and it is easy to see that the function
\begin{align*}
\frac{1}{n^{k-1}}\sum_{i=0}^{k-2} \frac{(k-2)!}{i!}(2\pi \sqrt{3}n)^i
\end{align*}
is decreasing as a function of $n$, so
\begin{align*}
I_1 - \frac{(k-2)!}{(4\pi m)^{k-1}} &\leq \frac{1}{(4\pi)^{k-1}} \sum_{n=\ell}^\infty |A_k^{(1)}(m,n)|^2 e^{-2\pi \sqrt{3}n} \frac{(k-2)!}{ \ell^{k-1}} \sum_{i=0}^\infty \frac{(2\pi \sqrt{3}\ell)^i}{i!} \\ 
&\leq \frac{(k-2)! e^{2\pi \sqrt{3}\ell}}{(4\pi\ell)^{k-1}} \sum_{n=\ell}^\infty |A_k^{(1)}(m,n)|^2 e^{-2\pi \sqrt{3}n}. \notag
\end{align*}
Using (\ref{Coeff Bound for A1}) gives
\begin{align*}
\sum_{n=\ell}^\infty |A_k^{(1)}(m,n)|^2 e^{-2\pi \sqrt{3}n} &\leq .242^2 \cdot 6.747^{2\ell} \cdot e^{-4\pi m \cdot 1.16} \sum_{n=\ell}^\infty e^{4\pi n \cdot .865-2\pi \sqrt{3}n} \\
&\leq (4.5763)(45.523)^\ell \cdot e^{-4\pi m \cdot 1.16} e^{-.01288 \ell}, \notag
\end{align*}
and thus
\begin{align*}
I_1 &\leq \frac{(k-2)!}{(4\pi m)^{k-1}} + \frac{(4.5763)(k-2)! (45.523)^\ell}{(4\pi \ell)^{k-1}} \cdot e^{-4\pi m \cdot 1.16} e^{10.86992 \ell}.
\end{align*}

\subsection{Upper Bound on $I_2$} Here we act on $F_{k,m}$ with $\alpha_2^{-1}$, and thus
\begin{align*}
F_{k,m}|_{\alpha_2^{-1}}(z) = \frac{1}{z^k} F_{k,m} \left(-\frac{1}{z} \right) = \frac{1}{z^k} \sum_{n=1}^\infty A_k^{(2)}(m,n) q^{n/2}.
\end{align*}
Putting this into the integral and rearranging, we have
\begin{align*}
I_2 &= \sum_{r,s \geq 1} A_k^{(2)}(m,r) \overline{A_k^{(2)}(m,s)} \int_{-1/2}^{1/2} e^{\pi i x (r-s)} \int_{\sqrt{1-x^2}}^\infty \frac{y^{k-2}}{(x^2+y^2)^k} e^{-\pi y(r+s)}dy dx.
\end{align*}
We easily obtain
\begin{align*}
I_2 &\leq \frac{1}{\pi} \left( \frac{2\sqrt{3}}{3} \right)^{k+2} \sum_{r,s \geq 1} |A_k^{(2)}(m,r)| |A_k^{(2)}(m,s)| \frac{e^{-\pi \sqrt{3}/2 (r+s)}}{r+s},
\end{align*}
and (\ref{Coeff Bound for A2}) yields
\begin{align*}
I_2 &\leq \frac{(108.842)^2}{\pi} \left( \frac{2\sqrt{3}}{3} \right)^{k+2} (6.269)^{2\ell} e^{-4\pi m \cdot 1.16} \sum_{r,s \geq 1} e^{\pi r \cdot .865} e^{\pi s \cdot .865} \frac{e^{-\pi \sqrt{3}/2 (r+s)}}{r+s}.
\end{align*}
It remains to bound the double sum. Since $r,s$ are positive integers, we have
\begin{align*}
\frac{1}{r+s} \leq \frac{1}{\sqrt{2rs}}
\end{align*}
so the double sum is bounded above by
\begin{align*}
\frac{1}{\sqrt{2}} \left(\sum_{n=1}^\infty \frac{e^{-\frac{1}{2}\sqrt{3}\pi n + \pi \cdot .865 n}}{n^{1/2}} \right)^2.
\end{align*}
We explicitly calculate the partial sum with 10000 terms and then bound the contribution of the tail, which gives
\begin{align*}
\sum_{n=1}^\infty \frac{e^{-\frac{1}{2}\sqrt{3}\pi n + \pi \cdot .865 n}}{n^{1/2}} &\leq 29.77087.
\end{align*}
Putting everything together, we obtain
\begin{align*}
I_2 &\leq 2363259 \left( \frac{2\sqrt{3}}{3} \right)^{k+2} (6.269)^{2\ell} e^{-4\pi m \cdot 1.16}.
\end{align*}

\subsection{Upper Bound on $I_3$}
We have
\begin{align*}
I_3 &= \int_{\mathcal{F}} F_{k,m}|_{\alpha_3^{-1}} (z) \overline{F_{k,m}|_{\alpha_3^{-1}} (z)} y^{k-2} dx dy,
\end{align*}
which implies
\begin{align*}
I_3 &\leq \frac{1}{\pi} \sum_{r,s \geq 1} |A_k^{(3)}(m,r)| |A_k^{(3)}(m,s)| \frac{e^{-\frac{1}{2}\sqrt{3}\pi (r+s)}}{r+s}.
\end{align*}
Applying (\ref{Coeff Bound for A3}) and proceeding as before, we have
\begin{align*}
I_3 &\leq 480 \cdot 65.766^{2\ell} \cdot e^{-4\pi m \cdot 1.16}.
\end{align*}

Putting everything together and simplifying, we have
\begin{align}\label{Upper Bound on Inner Product}
\langle F_{k,m}, F_{k,m} \rangle &\leq \frac{4(k-2)!}{(4\pi)^k m^{k-1}} + \frac{e^{2.908}(k-2)! (17.094)^k}{(4\pi)^{k}(\frac{k}{4}-1)^{k-1}}\cdot e^{-4\pi m \cdot 1.16} + e^{13.817} (1.828)^k e^{-4\pi m \cdot 1.16} \\ 
&\ \ \ + e^{5.03} (8.11)^k e^{-4\pi m \cdot 1.16}. \notag
\end{align}

\section{Proof of Theorem \ref{Main Theorem}}\label{proof of main thm section}
Let $n=n(k)$ be the dimension of $S_k(\text{SL}_2(\mathbb{Z}))$ and let $t = t(k)$ be the dimension of $S_k^{\text{new}}(2)$. Let $\{f_i\}_{i=1}^n$ be a basis of normalized Hecke eigenforms for $S_k(\text{SL}_2(\mathbb{Z}))$, and let $\{g_j\}_{j=1}^t$ be a basis of normalized newforms for $S_k^{\text{new}}(2)$. We would like to write $F_{k,m}$ as a linear combination of $f_i,f_i|V_2$, and $g_j$ (i.e. in terms of oldforms and newforms). Directly writing
\begin{align*}
F_{k,m} = \sum_i c_i f_i + \sum_i d_i f_i|V_2 + \sum_j e_j g_j
\end{align*}
with constants $c_i,d_i,e_j$ leads to problems, so we write $F_{k,m}$ with respect to a different basis.

To motivate our choice of basis we require some results on Petersson inner products. Specifically, we need to be able to evaluate $\langle f_i|V_2 , f_j | V_2 \rangle$ and $\langle f_i, f_j|V_2 \rangle$. We define a matrix $M$ by
\begin{align*}
M = \begin{pmatrix}
2 & 0 \\
0 & 1
\end{pmatrix}.
\end{align*}
Let $f_i,f_j$ be eigenforms in $S_k(\text{SL}_2(\mathbb{Z}))$. Then
\begin{align*}
\langle f_i|V_2 , f_j | V_2 \rangle &= 2^{-k} \langle f_i|_M , f_j |_M \rangle = 2^{-k} \langle f_i,f_j \rangle.
\end{align*}
This implies $\langle f_i|V_2 , f_j | V_2 \rangle = 0$ if $i \neq j$. Further, from Lemma 5 of \cite{Rou2} we have that
\begin{align*}
\langle f_i,f_i|V_2 \rangle = \frac{a_i(2)}{3 \cdot 2^{k-1}} \langle f_i, f_i \rangle,
\end{align*}
where $a_i(2)$ is the coefficient of $q^2$ in the Fourier expansion of $f_i$. Note that $a_i(2)$ is real since $f_i$ is a Hecke eigenform. An easy modification of the proof of this lemma shows that $\langle f_i,f_j|V_2 \rangle=0$ if $i \neq j$.

We would like to write $F_{k,m}$ with respect to a basis in which the forms are all orthogonal to each other with respect to the Petersson inner product, the $n$th coefficient of a basis element is bounded by $C d(n) n^{\frac{k-1}{2}}$ for some absolute constant $C \geq 1$, and the Petersson norm of each basis element is about the same order of magnitude. Using a basis of this form was suggested to us by Rouse (private communication). We choose our basis to consist of the forms $\{g_j\}_{j=1}^t, \{f_i|V_2\}_{i=1}^n, \{f_i-\kappa_if_i|V_2\}_{i=1}^n$, where $\kappa_i$ is chosen so that $f_i-\kappa_if_i|V_2$ is orthogonal to $f_i|V_2$. An easy calculation shows that $\kappa_i = \frac{2}{3} a_i(2)$. Now we write
\begin{align*}
F_{k,m} = \sum_i c_i \left(f_i - \kappa_i f_i|V_2 \right) + \sum_i d_i \left(2^{k/2} f_i|V_2 \right) + \sum_j e_j g_j.
\end{align*}
Note that we may take $c_i,d_i,e_j \in \mathbb{R}$ because $F_{k,m},f_i,g_j$ have real coefficients. Since the basis elements are orthogonal to each other we have
\begin{align*}
\langle F_{k,m}, F_{k,m} \rangle &= \sum_i c_i^2 \langle f_i - \kappa_i f_i|V_2, f_i - \kappa_i f_i|V_2 \rangle + \sum_i d_i^2 2^k \langle f_i|V_2, f_i | V_2 \rangle + \sum_j e_j^2 \langle g_j,g_j \rangle.
\end{align*}
Note that
\begin{align*}
\langle f_i - \kappa_i f_i|V_2, f_i - \kappa_i f_i|V_2 \rangle &= \langle f_i - \kappa_i f_i|V_2, f_i \rangle = \left(1 - \frac{4}{9} \frac{a_i^2(2)}{2^k} \right) \langle f_i,f_i \rangle \\
&\geq \frac{1}{9} \langle f_i, f_i \rangle, \notag
\end{align*}
where for the inequality we have used $|a_i(2)| \leq 2 \cdot 2^{\frac{k-1}{2}}$.
From work of Rouse \cite{Rou1} we have the inequality
\begin{align*}
\langle f_i,f_i \rangle &\geq \frac{6}{\pi^2} \cdot \frac{\Gamma(k)}{(4\pi)^k} \cdot \frac{1}{64 \log k},
\end{align*}
and from Lemma \ref{Newform Lower Bound} we have the inequality
\begin{align*}
\langle g_j, g_j \rangle &\geq \frac{6}{\pi^2} \cdot \frac{1}{1+\frac{1}{2}} \cdot \frac{\Gamma(k)}{(4\pi)^k} \cdot \frac{1}{86 \log k}.
\end{align*}
Taking these inequalities we see that
\begin{align*}
\langle F_{k,m}, F_{k,m} \rangle &\geq \sum_i \left(c_i^2 + d_i^2 \right) \frac{1}{9} \langle f_i,f_i \rangle + \sum_j e_j^2 \langle g_j,g_j \rangle \\
&\geq \sum_i (c_i^2 + d_i^2) \frac{1}{96\pi^2} \cdot \frac{\Gamma(k)}{(4\pi)^k \log k} + \sum_j e_j^2 \frac{2}{43 \pi^2} \cdot \frac{\Gamma(k)}{(4\pi)^k \log k} \notag \\
&\geq \left(\sum_i c_i^2 + \sum_i d_i^2 + \sum_j e_j^2 \right) \frac{1}{96\pi^2} \cdot \frac{\Gamma(k)}{(4\pi)^k \log k}. \notag
\end{align*}
This obviously implies that
\begin{align*}
\sum_i c_i^2 + \sum_i d_i^2 + \sum_j e_j^2 &\leq \langle F_{k,m}, F_{k,m} \rangle \cdot 96 \pi^2 \cdot \frac{(\log k) (4\pi)^k}{(k-1)!}.
\end{align*}
Recall that $n$ is the dimension of $S_k(\text{SL}_2(\mathbb{Z}))$ and $t$ is the dimension of $S_k^{\text{new}}(2)$. We have $n \leq \frac{k}{12}$, and specializing Theorem 1 of \cite{Mar} gives
\begin{align*}
t &= k - 1 - \left\lfloor \frac{k}{4} \right\rfloor - 2\left\lfloor \frac{k}{3} \right\rfloor \leq 2 + \frac{k}{12}.
\end{align*}
By the Cauchy-Schwarz inequality, we have
\begin{align*}
\sum_{i=1}^n |c_i| + \sum_{i=1}^n |d_i| + \sum_{j=1}^t |e_j| &\leq \sqrt{2n + t} \cdot \sqrt{\sum_i c_i^2 + \sum_i d_i^2 + \sum_j e_j^2} \\
&\leq \sqrt{\frac{k}{4}+2} \cdot \sqrt{\sum_i c_i^2 + \sum_i d_i^2 + \sum_j e_j^2}. \notag
\end{align*}
Using the triangle inequality and (\ref{Upper Bound on Inner Product}) we obtain
\begin{align*}
\sum_{i=1}^n |c_i| + \sum_{i=1}^n |d_i| + \sum_{j=1}^t |e_j| &\leq \sqrt{\log k} \Bigg(\frac{44}{m^{\frac{k-1}{2}}} + \frac{e^{4.601}(6.274)^k}{(\frac{k}{4}-1)^{\frac{k-1}{2}}} e^{-2\pi m \cdot 1.16} \\
&+ \frac{e^{10.057}(4.793)^k}{\sqrt{(k-2)!}} e^{-2\pi m \cdot 1.16} + \frac{e^{5.663}(10.096)^k}{\sqrt{(k-2)!}}e^{-2\pi m \cdot 1.16} \Bigg). \notag
\end{align*}
We must determine the absolute constant $C \geq 1$ such that the $n$th coefficients of $g_j, 2^{k/2}f_i|V_2$, and $f_i - \kappa_i f_i|V_2$ are bounded above by $C d(n) n^{\frac{k-1}{2}}$. The $g_j$ are newforms so their coefficients are bounded above by $d(n) n^{\frac{k-1}{2}}$. Let $f_i$ be as above with Fourier expansion given by
\begin{align*}
f_i(z) = \sum_{n=1}^\infty a_i(n) q^n.
\end{align*}
Then $f_i|V_2 = \sum_{n=1} a(n) q^{2n}$, so the $2n$th coefficient of $2^{k/2} f_i|V_2$ is bounded above by
\begin{align*}
\left|2^{k/2} a_i(n) \right| &\leq 2^{k/2} d(n) n^{\frac{k-1}{2}} = \sqrt{2} \ d(n) (2n)^{\frac{k-1}{2}} \leq \sqrt{2} \ d(2n) (2n)^{\frac{k-1}{2}}.
\end{align*}
Arguing similarly, we see that the coefficients of $f_i - \kappa_i f_i|V_2$ are bounded above by $\frac{7}{3} d(n) n^{\frac{k-1}{2}}$. Hence we may take $C = \frac{7}{3}$. This implies that the absolute value of the coefficient of $q^n$ in $F_{k,m}$ is bounded above by
\begin{align*}
\sqrt{\log k} \Bigg(\frac{103}{m^{\frac{k-1}{2}}} &+ \frac{e^{5.449}(6.274)^k}{(\frac{k}{4}-1)^{\frac{k-1}{2}}} e^{-2\pi m \cdot 1.16} \\
&+ \frac{e^{10.905}(4.793)^k}{\sqrt{(k-2)!}} e^{-2\pi m \cdot 1.16} + \frac{e^{6.511}(10.096)^k}{\sqrt{(k-2)!}}e^{-2\pi m \cdot 1.16} \Bigg) d(n) n^{\frac{k-1}{2}}. \notag
\end{align*}
Now define $B(k)$ by
\begin{align*}
B(k) &= \frac{e^{5.449}(6.274)^k}{(\frac{k}{4}-1)^{\frac{k-1}{2}}} + \frac{e^{10.905}(4.793)^k}{\sqrt{(k-2)!}} + \frac{e^{6.511}(10.096)^k}{\sqrt{(k-2)!}},
\end{align*}
so that the $n$th coefficient of $F_{k,m}$ is bounded above in absolute value by
\begin{align*}
\sqrt{\log k} \left(\frac{103}{m^{\frac{k-1}{2}}} + B(k)e^{-2\pi m \cdot 1.16} \right) d(n) n^{\frac{k-1}{2}}.
\end{align*}
Let $G \in S_k(2)$ be given by
\begin{align*}
G(z)  = \sum_{m=1}^{\ell - 1} a(m) F_{k,m} = \sum_{n=1}^\infty a(n) q^n.
\end{align*}
Applying the triangle inequality gives
\begin{align*}
|a(n)| &\leq \sqrt{\log k} \left(103\sum_{m=1}^{\ell-1} \frac{|a(m)|}{m^{\frac{k-1}{2}}} + B(k) \sum_{m=1}^{\ell -1} |a(m)| e^{-7.288m} \right) d(n) n^{\frac{k-1}{2}},
\end{align*}
which yields Theorem \ref{Main Theorem}.

\subsection*{Acknowledgements} The authors thank Jeremy Rouse for very helpful conversations. This work was partially supported by a grant from the Simons Foundation (\#281876 to Paul Jenkins). The
second author thanks the Brigham Young University Mathematics Department for support of this research, and the ORCA Mentoring Environment Grant for financial assistance.

\bibliographystyle{amsplain}

\end{document}